\documentclass[12pt]{amsart}
\usepackage{amsfonts}
\usepackage{amsmath}
\usepackage{soul}
\usepackage{color}
\usepackage{amssymb}
\usepackage{graphicx}
\usepackage[all]{xy}\SelectTips{cm}{}
\definecolor{lightblue}{rgb}{.30,.95,.2}
\sethlcolor{lightblue}
\usepackage{epsf,epsfig}

\newcommand{\alga}{\mathbf A}

\providecommand{\U}[1]{\protect\rule{.1in}{.1in}}
\newtheorem{theorem}{Theorem}
\theoremstyle{plain}

\newtheorem{corollary}{Corollary}

\newtheorem{definition}{Definition}
\newtheorem{example}{Example}

\newtheorem{lemma}{Lemma}

\newtheorem{proposition}{Proposition}

\numberwithin{equation}{section}

\usepackage
[authormarkup=none]
{changes}

\definechangesauthor[color=red]{Ref1}

\begin{document}
\title{Residuated Relational Systems}
\author[Bonzio S.]{S. Bonzio}
\address{Stefano Bonzio, University of Cagliari\\
Italy}
\email{stefano.bonzio@gmail.com}

%
%
\author[Chajda I.]{I. Chajda}
\address{Ivan Chajda, Palack\'y University Olomouc\\
Czech Republic}
\email{ivan.chajda@upol.cz}

\begin{abstract}

The aim of the present paper is to generalize the concept of residuated poset, by replacing the usual partial ordering by a generic binary relation, giving rise to relational systems which are residuated. In particular, we modify the definition of adjointness in such a way that the ordering relation can be harmlessly replaced by a binary relation. By enriching such binary relation with additional properties we get interesting properties of residuated relational systems which are analogical to those of residuated posets and lattices.
\end{abstract}
\thanks{{\bf Corresponding author}: Stefano Bonzio,  stefano.bonzio@gmail.com}
\maketitle

\section{Introduction}
The study of binary relations traces back to the work of J. Riguet \cite{Rig48}, while a first attempt to provide an algebraic theory of relational systems is due to Mal'cev \cite{Malcev}. Relational systems of different kinds have been investigated by different authors for a long time, see for example \cite{Belo02}, \cite{Bonzio15}, \cite{ChLa13}, \cite{ChLa14}, \cite{ChLa16}. Binary relational systems are very important for the whole of mathematics, as relations, and thus relational systems, represent a very general framework appropriate for the description of several problems, which can turn out to be useful both in mathematics and in its applications. For these reasons, it is fundamental to study relational systems from a structural point of view. In order to get deeper results meeting possible applications, we claim that the usual domain of binary relations shall be expanded. More specifically, our aim is to study general binary relations on an underlying algebra whose operations interact with them. \\
A motivating idea comes from the concept of \emph{polarity} introduced by Birkhoff, see \cite{Birk61}. In detail, consider a binary relation $R$ on a set $A$ (i.e.\ $R\subseteq A\times A$). For any subset $X\subseteq A$, we can define the sets
$$X^{\ast}= \{ y\in A: (x,y)\in R, \text{ for each  } x\in X\}, $$
$$X^{\dagger}= \{ x\in A: (x,y)\in R, \text{ for each  } y\in X\}. $$
When considering the power set $\mathcal{P}(A)$, we denote by $f$ and $g$ the mappings $f(X)= X^{\ast}$, $g(X)= X^{\dagger}$. \\
Following \cite{Birk61}, we say that the pair $(f,g)$ forms a \emph{polarity}, if, for every $X_1,X_1\subseteq A$, we have $X_1\subseteq g(X_2)$ if and only if $X_2\subseteq f(X_1)$.
We can freely consider two mappings $f,g$ on a non-void set $A$ into itself and a binary relation $R$ on $A$ and say that the pair $(f,g)$ forms a \emph{Galois connection} whenever $$(x,g(y))\in R \text{ if and only if } (f(x),y)\in R,$$
for any $x,y\in A$. \\
In order to pursue the idea of extending the study of binary relations from sets to algebras we define the notion of a Galois connection on an algebra equipped with an additional binary relation. Let $\mathbf{A}=\langle A,\cdot,\rightarrow\rangle$ be an algebra of type $\langle 2,2\rangle$. For a given element $y\in A$, we define the two mappings 
$f_{y}(x)=x\cdot y, \;\;\; g_{y}(x)=y\rightarrow x.$
We say that the pair $(f_y , g_y)$ is a \emph{residuated pair} if it forms a \emph{Galois connection}, i.e. $$(x,y\rightarrow z)\in R  \text{  if and only if  } (x\cdot y, z)\in R.$$
We will show that this approach may enrich the study of binary relations in general on one hand, and, most importantly, the study of residuated structures on the other. Indeed, when the relation $R$ is taken to be a partial ordering, we get a residuated poset, which is an important tool both in pure algebra and in the algebraic studies in logic. \\
Our idea is that it makes sense to study the cases where such relation on the residuated algebra need not be a partial order, but it can be a weaker relation. This motivates us to develop a general setting for residuated structures, which is a step towards a generalization of the theory of (commutative) residuated lattices and of ordered (commutative) residuated monoids. \\


The paper is structured as follows. In Section \ref{sec:3.1} the notion of a residuated relational system is introduced and the basic properties are proved. In Section \ref{sec:3.2} we develop the concept of a pre-ordered residuated system, which is nothing but a residuated relational system whose relation is reflexive and transitive; finally, in Section \ref{sec:3.3}, we expand the notion of a residuated relational system by adding negation.


\section{Residuated relational systems}\label{sec:3.1}
We begin by introducing the central notion that will be used throughout the paper. 

\begin{definition}\label{def: residuated relational systems}
A \emph{residuated relational system} is a structure $ \mathbf{A}=\langle A, \cdot, \rightarrow, 1, R\rangle $, where $ \langle A, \cdot, \rightarrow, 1\rangle $ is an algebra of type $ \langle 2,2,0\rangle $ and $R$ is a binary  relation on $ A $ and satisfying the following properties:   
\begin{itemize}
\item[1)] $ \langle A,\cdot, 1\rangle $ is a commutative monoid; 
\item[2)] $ (x,1)\in R $, for each $ x\in A $; 
\item[3)] $ (x\cdot y,z)\in R $ if and only if $ (x,y\rightarrow z)\in R $.
\end{itemize}
\end{definition}
\noindent
We will refer to the operation $ \cdot $ as multiplication, to $ \rightarrow $ as its \emph{residuum} and to condition 3) as \emph{residuation}. 

\begin{example}\label{ex: POCRIMs}
\emph{Any commutative residuated integral pomonoid (see \cite{BlokRaft97} for details) - pocrim for short - is an example of residuated relational system, where $R$ coincides with a partial order.} 
\end{example}
\begin{example}
\emph{Any (commutative) residuated lattice is a residuated relational system, where $R$ is a partial lattice order.}
\end{example}

Multiplication, as well as its residuum, can be defined as residuated maps on $ A $. More precisely, let $ \mathbf{B}=\langle B,R_1\rangle $ and $ \mathbf{C}=\langle C,R_2 \rangle $ be two relational systems (sets with a binary relation), we say that a map $ f:B\rightarrow C $ is \textit{residuated} if there exists a map $ g:C\rightarrow B $, such  that $ (f(b),c)\in R_2 $ if and only if $ (b,g(c))\in R_1 $. The two maps, $ f $ and $ g $, form a \textit{pair} of residuated maps. Setting $ \mathbf{A}=\mathbf{B}=\mathbf{C} $ and defining for any $ a\in A $, $ f_{a}(x)= x\cdot a $ and $ g_{a}(x)=a\rightarrow x $ we obtain that  the two maps $ f_{a} $ and $ g_{a} $ form a residuated pair. 

It is useful to recall here the notions of upper cone (with respect to a pair of elements) and of supremal element.
\begin{definition}\label{def:upper cone}
For any elements $a,b\in A$, the \emph{upper cone} of $a,b$ is the set
\[
U_{R}(a,b)=\{c\in A:(a,c)\in R\text{ and }(b,c)\in R\}.
\]
\end{definition}
\noindent
It is immediate to notice that in a residuated relational system, it may never be the case that $ U_{R}(a,b)=\emptyset $ for any $ a,b\in A $, as, by condition 2) in Definition \ref{def: residuated relational systems}, $ 1\in U_{R}(a,b) $.  
\begin{definition}\label{def: supremal element}
 An element $ w\in U_{R}(a,b) $ is a \emph{supremal element} for $ a,b $ if for each $ z\in U_{R}(a,b) $, with $ z\neq w $, then $ (w,z)\in R $. 
\end{definition}
\noindent
Obviously, whenever $ R $ is a lattice order relation on $ A $, then the supremal element for $ a,b \in A $ always exists, is unique and coincides with $\sup(a,b) $.  The definition of supremal element can be easily extended to subsets of $ A $. Let $ Z\subseteq A $, an element $ k\in A $ is a \emph{supremal element for Z} if $ (z,k)\in R $, for each $ z\in Z $ and for each $ w\in A $ with $w\neq k$ and $(z,w)\in R$ for all $z\in Z$ we have $(k,w)\in R $. 
In case $ R $ coincides with a partial ordering and sup $Z $ exists then sup $ Z $ is the unique supremal element for $ Z $. Notice that for a residuated relational system equipped with an arbitrary binary relation $ R $, a supremal element for a subset $ Z $ need not exist, and if it does, it need not be unique. 


The basic properties for residuated relational systems are subsumed in the following:

\begin{proposition}\label{prop: aritmetica}
Let $ \mathbf{A}=\langle A, \cdot, \rightarrow, 1, R\rangle $ be a residuated system, then
\begin{itemize}
\item[(a)] If $ x\rightarrow y=1 $ then $ (x,y)\in R $, for all $x,y\in A$.
\item[(b)] $ (x,1\rightarrow 1)\in R $, for each $ x\in A $.
\item[(c)] $(1,x\rightarrow 1)\in R $, for each  $x\in A $.
\item[(d)] If $ x\rightarrow y=1 $ then $(z\cdot x,y)\in R$, for all $x,y,z\in A$.
\item[(e)] $(x, y\to 1)\in R$, for all $x,y\in A$.
\end{itemize}
\end{proposition}
Recall that a binary relation $ R $ is said to be \emph{antisymmetric} whenever, if $ (x,y)\in R $ and $ (y,x)\in R $ then $ x = y $. 
The next proposition shows that a residuated relational system whose relation is antisymmetric turns into an algebra of type $ \langle 2,2,0\rangle $.
\begin{proposition}\label{prop: definitione equazioanle di R}
Let $ \mathbf{A}=\langle A, \cdot, \rightarrow, 1, R\rangle $ be a residuated relational system, with $ R $ an antisymmetric relation. Then 
\begin{itemize}
\item[i)] $ (x,y)\in R $ if and only if $ x\rightarrow y = 1 $.
\item[ii)] If $ R $ is also reflexive, then $ (x\cdot y, y)\in R $ and $ (x\cdot y,x)\in R $
\end{itemize}
\end{proposition}
\proof
i) One direction holds by Proposition \ref{prop: aritmetica}. For the converse, suppose $ (x,y)\in R $; then $ (1,x\rightarrow y)\in R $, by residuation. On the other hand, by condition 2) in Definition \ref{def: residuated relational systems}, $ (x\rightarrow y,1)\in R $, and since $ R $ is antisymmetric, it follows that $ x\rightarrow y = 1 $. \\
ii) By reflexivity of $ R $, $ (y,y)\in R $, thus $ y\rightarrow y = 1 $, by i). Since $ (x,1)\in R $, we have $ (x,y\rightarrow y)\in R $ and, by residuation, $ (x\cdot y, y)\in R $. The other claiming is proved analogously using commutativity of multiplication.  
\endproof 

\begin{proposition}\label{prop:aritmetica2}
Let $ \alga $ be a residuated relational system with a reflexive relation $ R $. Then for all $x,y\in A $
\begin{itemize}
\item[(a)] $ (1,x\rightarrow x)\in R $.
\item[(b)] $ ((x\rightarrow y)\cdot x, y)\in R $.
\item[(c)] $ (x,y\rightarrow x\cdot y)\in R $.
\item[(d)] $ (x,1\rightarrow x)\in R  $ and $ (1\rightarrow x, x)\in R $.
\item[(e)] $ (x,(x\rightarrow y)\rightarrow y)\in R $. 
\end{itemize} 
\end{proposition}
\begin{proof}
(a) By reflexivity $ (x,x) = (1\cdot x, x)\in R $, thus, by residuation $ (1,x\rightarrow x)\in R $.\\
(b) $ (x\rightarrow y, x\rightarrow y)\in R $ yields $ ((x\rightarrow y)\cdot x, y)\in R $. \\
(c) $ (x\cdot y,x\cdot y)\in R $ clearly implies $ (x,y\rightarrow x\cdot y)\in R $. \\
(d) $ (x,x) = (x\cdot 1, x)\in R $ implies $ (x,1\rightarrow x)\in R  $. Similarly, $ (1\rightarrow x, x)\in R $ is obtained by residuation from $ (1\rightarrow x, 1\rightarrow x)\in R $. \\
(e) By (b) and commutativity of multiplication we have $ (x\cdot (x\rightarrow y), y)\in R $, hence $ (x,(x\rightarrow y)\rightarrow y)\in R $.
\end{proof}

Residuated relational systems are introduced to be a generalization of well-known structures as (integral, commutative) residuated lattices and (integral) residuated pomonoids. Still, the aim of the present work is handling with ``genuine'' residuated relational systems, namely systems that cannot be directly turned into an algebra. For this reason, it shall be clear why we will not concentrate our analysis on those systems whose relation is antysimmetric. The most fruitful results can then be reached considering systems, whose relation $ R $ is a pre-order. 

\section{Pre-ordered residuated systems}\label{sec:3.2}
Recall that a pre-order relation $\preceq $ on a set $ A $ is a binary relation which is reflexive and transitive. 
Two elements $ a,b $ in a pre-ordered set $ A $ are \emph{incomparable}, in symbols $ a\parallel b $, if $ a\npreceq b $ and $ b\npreceq a $.
It follows that the relation of incomparability is symmetric. 
\begin{definition}\label{def: sistema res. pre-ordinato}
A \emph{pre-ordered residuated system} is a residuated relational system $ \mathbf{A}=\langle A, \cdot, \rightarrow, 1, \preceq \rangle $, where $ \preceq $ is a pre-order on $ A $. 
\end{definition}


\noindent
By convention we will write $U_{\preceq}(x) $ instead of $ U_{\preceq}(x,x) $. It readily follows, by transitivity of $ \preceq $, that if $ x\preceq y $ then $ U_{\preceq}(y)\subseteq U_{\preceq}(x) $.  \\
The following proposition shows the basic properties of pre-ordered residuated systems.

\begin{proposition}\label{pro: aritmetica quasi-order}
Let $ \mathbf{A} $ be a pre-ordered residuated system. 
Then 
\begin{itemize}
\item[(a)] $ \cdot $ preserves the pre-order in both positions
\item[(b)] $ x\preceq y $ implies $ y\rightarrow z\preceq x\rightarrow z $ and $ z\rightarrow x\preceq z\rightarrow y $
\item[(c)] $ x\cdot (y\rightarrow z) \preceq y\rightarrow x\cdot z $
\item[(d)] $ x\cdot y\rightarrow z \preceq x\rightarrow (y\rightarrow z) $ 
\item[(e)] $ x \rightarrow (y\rightarrow z)\preceq x\cdot y\rightarrow z $
\item[(f)] $ x\rightarrow (y\rightarrow z)\preceq y\rightarrow (x\rightarrow z) $
\item[(g)] $ (x\rightarrow y)\cdot(y\rightarrow z)\preceq x\rightarrow z $
\item[(h)] $ x\cdot y\preceq y,x $
\item[(i)] $ x\rightarrow y\preceq (y\rightarrow z)\rightarrow (x\rightarrow z) $
\end{itemize}
\end{proposition}
\proof
(a) Suppose $ x\preceq y $. Since $ \preceq $ is reflexive, $ y\cdot z\preceq y\cdot z $, hence $ y\preceq z\rightarrow (z\cdot y) $. Then, by transitivity, we get $ x\preceq z\rightarrow (z\cdot y) $, therefore, by residuation, $ x\cdot z\preceq y\cdot z $. Preservation of the pre-order in both positions follows trivially by commutativity of multiplication. \\
(b) Let $ x\preceq y $, then $ x\cdot(y\rightarrow z)\preceq y\cdot(y\rightarrow z)\preceq z $, where we have used (a) and then commutativity, residuation and reflexivity of $ \preceq $. By transitivity, $ x\cdot(y\rightarrow z)\preceq z $, i.e. $ y\rightarrow z\preceq x\rightarrow z $ by residuation (and commutativity). For the remaining claim, residuation and reflexivity of $ \preceq $ guarantee that $ z\cdot (z\rightarrow x) \preceq x $, hence, by transitivity, $ z\cdot (z\rightarrow x) \preceq y $, therefore $ z\rightarrow x \preceq z\rightarrow y $.  \\
(c) By Proposition \ref{prop:aritmetica2} (b), $ (y\rightarrow z)\cdot y\preceq z $ and by Proposition \ref{prop:aritmetica2} (c), $ z\preceq x\rightarrow z\cdot x$, thus $ (y\rightarrow z)\cdot y\preceq x\rightarrow z\cdot x  $ and, using residuation (twice) and commutativity, we have $ x\cdot (y\rightarrow z) \preceq y\rightarrow x\cdot z $.   \\
(d) is proved as follows:
\begin{align*}
& (x\cdot y\rightarrow z)\cdot (x\cdot y)\preceq z && (\text{Proposition \ref{prop:aritmetica2}})
\\ & ((x\cdot y\rightarrow z)\cdot x)\cdot y\preceq z && (\text{Associativity})
\\ & (x\cdot y\rightarrow z)\cdot x\preceq y \rightarrow z && (\text{Residuation})
\\ & x\cdot y\rightarrow z \preceq x\rightarrow (y\rightarrow z) && (\text{Residuation})
\end{align*}
(e) is proved similarly. 
\begin{align*}
& x\rightarrow (y\rightarrow z) \preceq x\rightarrow (y\rightarrow z) &&
\\ & (x\rightarrow (y\rightarrow z))\cdot x \preceq y\rightarrow z && (\text{Residuation})
\\ & ((x\rightarrow (y\rightarrow z))\cdot x)\cdot y \preceq z && (\text{Residuation})
\\ & (x\rightarrow (y\rightarrow z))\cdot (x\cdot y) \preceq z && (\text{Associativity})
\\ & x\rightarrow (y\rightarrow z)\preceq (x\cdot y) \rightarrow z && (\text{Residuation})
\end{align*}
(f) Using d) and (e), we have $ x\rightarrow (y\rightarrow z)\preceq x\cdot y\rightarrow z = y\cdot x\rightarrow z\preceq y\rightarrow (x\rightarrow z) $. \\
(g) By Proposition \ref{prop:aritmetica2} and commutativity, $ x\cdot (x\rightarrow y)\preceq y $, hence $ x\cdot (x\rightarrow y)\cdot(y\rightarrow z)\preceq y\cdot(y\rightarrow z)\preceq z $, thus, by residuation, $ (x\rightarrow y)\cdot(y\rightarrow z)\preceq x\rightarrow z $. \\
(h) Since $ x\preceq 1 $, we have $ x\cdot y\preceq 1\cdot y = y $; similarly for $ x\cdot y\preceq x $. \\
(i) follows from (g) using residuation. 
\endproof
\noindent
In the following result we give some necessary and sufficient conditions for a generic structure $ \langle A, \cdot ,\rightarrow , 1, \preceq\rangle $ to be effectively a pre-ordered residuated system. 
\begin{theorem}\label{th: A è residuato sse}
Let $\alga =\langle A, \cdot ,\rightarrow, 1, \preceq\rangle $ be a quintuple such that $ \cdot $ and $ \rightarrow $ are binary operations on $A$, $ \preceq $ is a binary relation on $ A $ and $ 1\in A $. Then $ \alga $ is a pre-ordered residuated system if and only if it satisfies the following conditions: 
\begin{itemize}
\item[(a)] $ \langle A,\cdot, 1\rangle $ is a commutative monoid
\item[(b)] $ \preceq $ is a pre-order on $ A $ such that $ x\preceq 1 $ for each $ x\in A $
\item[(c)] $ x\cdot y\rightarrow z\preceq x\rightarrow (y\rightarrow z) $ and $ x\rightarrow (y\rightarrow z)\preceq x\cdot y\rightarrow z $ for each $x,y,z\in A$.
\item[(d)] $ x\preceq y$ if and only if $ 1\preceq x\rightarrow y $ for each $x,y\in A$.
\end{itemize}

\end{theorem}
\begin{proof}
Suppose that $ \alga $ is a pre-ordered residuated system, then (a) and (b) hold by definition; (c) follows from Proposition \ref{pro: aritmetica quasi-order} and (d) is easily derived using residuation. \\
Conversely, assume $ \alga $ satisfies conditions (a) to (d). We only have to derive residuation to get a residuated relational system. Assume $ x\cdot y\preceq z $. By conditions (d) and (c), we have $  1\preceq x\cdot y\rightarrow z\preceq x\rightarrow (y\rightarrow z) $ and, since $ \preceq $ is transitive, $  1\preceq x\rightarrow (y\rightarrow z) $, thus, using (d) $ x\preceq y\rightarrow z $. On the other hand, assume $ x\preceq y\rightarrow z $, then, by (d), $ 1\preceq x\to (y\rightarrow z)\preceq x\cdot y\rightarrow z $, by (c). Due to transitivity and (d) we have $ x\cdot y\preceq z $. 
\end{proof}

The concept of \emph{directoid} has been originally introduced by Je\v{z}ek and Quackenbush \cite{JazekQ90}. A comprehensive and detailed exposition of the theory of directoids can be found in \cite{Chajdabook}, \cite{CGKGLP} and \cite{Chajda07}. Basically, directoids are the algebraic counterpart of directed partially ordered sets. 

Following the same ideas, we can think of capturing some properties of pre-ordered residuated systems by associating them to algebraic structures. We therefore introduce a binary operation on a pre-ordered residuated system as follows:
\begin{definition}\label{def: pseudo-join}
Let $ \mathbf{A}=\langle A, \cdot, \rightarrow, 1, \preceq \rangle $ be a pre-ordered residuated system. We define the following binary operation $ \sqcup $ on $ A $ as follows: 
\begin{itemize}
\item[i)] If $ x\preceq y $ then $ x\sqcup y=y $; 
\item[ii)] If $ x\npreceq y $ and $ y\preceq x $ then $ x\sqcup y = y\sqcup x = x $;
\item[iii)] If $ x\parallel y $ then $ x\sqcup y =y\sqcup x\in U_{\preceq}(x,y) $ is chosen arbitrarily.  
\end{itemize}

\end{definition} 
\noindent

The following elementary fact holds in any pre-ordered residuated system equipped with a binary operation defined as in Definition \ref{def: pseudo-join}.
\begin{lemma}\label{lem: x sotto x,y}
Let $ \mathbf{A} $ be a pre-ordered residuated system and $ \sqcup $ a binary operation on $ A $, defined as in \emph{Definition \ref{def: pseudo-join}}. Then for any $ x,y\in A $, $ x\preceq x\sqcup y $ and $ y\preceq x\sqcup y $. 
\proof
For any $ x,y\in A $, the following cases may arise: 
\begin{enumerate}
\item $ x\preceq y $, then $ x\sqcup y = y $ and clearly $ x,y\preceq x\sqcup y $. 
\item $ x\npreceq y $ and $ y\preceq x $, then $ x\sqcup y = x $, hence by reflexivity of $ \preceq $, $ x,y\preceq x\sqcup y $. 
\item $ x\parallel y $, then $ x,y\preceq x\sqcup y $, since $ x\sqcup y\in U_{\preceq}(x,y) $. 
\end{enumerate}
\endproof
\end{lemma}
\noindent
The above lemma expresses the intuitive fact that for any elements $ x,y\in A $, $ x\sqcup y\in U_{\preceq}(x,y) $.
\begin{definition}\label{def: quasi-directoid}
An algebra $ \alga = \langle A, \sqcup\rangle $ of type $ \langle 2\rangle $ is  called a \emph{quasi-directoid} if it satisfies: 
\begin{itemize}
\item[a)] $ x\sqcup x = x $; 
\item[b)] $ x\sqcup (x\sqcup y) = x\sqcup y $, $y\sqcup (x\sqcup y) = x\sqcup y$;
\item[c)] $ x\sqcup ((x\sqcup y)\sqcup z)= (x\sqcup y)\sqcup z $;
\end{itemize}
\end{definition}
\noindent
Now we can give an algebraic counterpart to the concept of pre-ordered residuated system. 

\begin{definition}\label{def:pre-ordered residuated quasi-lattices}
A \emph{residuated quasi-directoid} is an algebra $ \mathbf{A}=\langle A, \cdot, \rightarrow, \sqcup, 1\rangle  $ of type $ ( 2,2,2,0)  $ such that the term reduct $ \langle A, \sqcup\rangle $ is a quasi-directoid satisfying also the following axioms: 
\begin{itemize}
\item[e)] $ \langle A, \cdot, 1\rangle $ is a commutative monoid;
\item[f)] $ x\sqcup 1 = 1 $;
\item[g)] $ (x\cdot y)\sqcup z = z $ if and only if $ x\sqcup (y\rightarrow z)= y\rightarrow z $.
\end{itemize}
\end{definition}
\noindent
The terminology introduced in the definition above stresses the similarities with directoids. Indeed the term reduct $ \langle A, \sqcup \rangle $ is not very different from a directoid: any directoid satisfies identities a), b) and c), however, in general, the quasi-directoid does not satisfy $ (x\sqcup y) \sqcup x = x\sqcup y $. We will refer to the operation $ \sqcup $ as quasi-join. \\
Quasiidentity g) expresses a condition of residuation, namely the operation $ \rightarrow $ can be interpreted as the residuum of multiplication.

It is our aim to show a correspondence between pre-ordered residuated systems and residuated quasi-directoids, so that it will appear clear that the latter represent the algebraic counterpart of the former.

\begin{theorem}\label{th: algebra indotta dal sistema preordinato}
Let $\mathbf{A}=\langle A, \cdot, \rightarrow, 1, \preceq \rangle$ be a pre-ordered residuated system. Then, by defining a binary operation $\sqcup$ according to Definition \ref{def: pseudo-join}, the algebra $ \langle A, \cdot, \rightarrow, \sqcup , 1 \rangle $ is a residuated quasi-directoid.
\proof
We proceed by checking that $ \langle A, \cdot, \rightarrow, \sqcup , 1 \rangle $ satisfies all the conditions in Definition \ref{def:pre-ordered residuated quasi-lattices}. \\
e) trivially follows from the assumption that $ \mathbf{A} $ is a pre-ordered residuated system. \\
f) $ x \sqcup 1 = 1 $ since $ x\preceq 1 $ for each $ x\in A $. \\
g) follows trivially from the fact that $ \mathbf{A} $ is a pre-ordered residuated system. \\
Let us now check that the reduct $ \langle A, \sqcup \rangle $ is a quasi-directoid. \\
a) $ x\sqcup x = x $ since $ \preceq $ is reflexive. \\
b) We proceed through a case-splitting argument. \\
Case 1: Assume $ x\preceq y $. Then by Definition \ref{def: pseudo-join}, $ x\sqcup y = y $, hence $ x\sqcup(x\sqcup y) = x\sqcup y$ and $ y\sqcup (x\sqcup y) = y\sqcup y = y = x\sqcup y $. \\
Case 2: Assume $ x\npreceq y $ and $ y\preceq x $. Hence $ x\sqcup y=y\sqcup x= x $. Then $ x\sqcup (x\sqcup y) = x\sqcup x = x = x\sqcup y $ and $ y\sqcup (x\sqcup y)= y\sqcup x = x = x\sqcup y $. \\
Case 3: Assume $ x\npreceq y $ and $ y\npreceq x $. Then $ x\sqcup y\in U_{\preceq}(x,y) $. Since $ y\preceq x\sqcup y $ and $ x\preceq x\sqcup y$, by Lemma \ref{lem: x sotto x,y}, we get that $ x\sqcup (x\sqcup y) = x\sqcup y $ and $ y\sqcup (x\sqcup y) = x\sqcup y $. \\
c) As for b), we consider all the possible cases that may arise. \\
Case 1: Assume $ x\preceq y $. The left-hand side of equation c) reads $ x\sqcup ((x\sqcup y)\sqcup z)= x\sqcup (y\sqcup z) = y\sqcup z $, since $ x\preceq y\preceq y\sqcup z $, by Lemma \ref{lem: x sotto x,y}. Similarly, under this assumption, the right-hand side reads $ (x\sqcup y)\sqcup z = y\sqcup z $. \\
Case 2: Assume $ x\npreceq y $ and $ y\preceq x $. Then we have $ x\sqcup ((x\sqcup y)\sqcup z)= x\sqcup (x\sqcup z)= x\sqcup z $, by Lemma \ref{lem: x sotto x,y}. On the other hand, the right-hand side reads $ (x\sqcup y)\sqcup z = x\sqcup z $. \\
Case 3: Assume $ x\npreceq y $ and $ y\npreceq x $. Then, by definition, $ x\sqcup y= y\sqcup x= w $, for a certain $ w\in U_{\preceq}(x,y)$. Therefore, the left-hand side of equation c) is $ x\sqcup ((x\sqcup y)\sqcup z)= x\sqcup (w\sqcup z) = w\sqcup z $, as $ x\preceq w\preceq w\sqcup z $. The right-hand side reads $ (x\sqcup y)\sqcup z = w\sqcup z $. 
\endproof
\end{theorem}
\noindent

It shall be pointed out that in general, any directed relational system can be associated to more than one quasi-directoid, since for each pair of incomparable elements $ x,y $, the element $ x\sqcup y $ is not uniquely determined in the upper cone of the two elements. 

Following the same idea developed in \cite{ChLa13} and in the previous section, we can define a relation $ \preceq_{I} $, \emph{induced} by a quasi-directoid $ \mathbf{A} $, as follows: 
\begin{equation}\label{eq: relazione indotta}
x\preceq_{I} y  \;\; \text{if and only if} \;\; x\sqcup y = y.
\end{equation}
Given a residuated quasi-directoid $ \mathbf{A} $, we refer to the relational system $ \langle A, \cdot, \rightarrow, 1, \preceq_{I} \rangle $, as to the induced relational system. 

We can also prove a converse statement of Theorem \ref{th: algebra indotta dal sistema preordinato}, i.e. that the relational system induced by a residuated quasi-directoid is actually a pre-ordered residuated system. 
\begin{theorem}\label{th: dal quasi-lattice al pre-order system}
Let $ \mathbf{A}= \langle A, \cdot, \rightarrow, \sqcup, 1\rangle  $ be a residuated quasi-directoid and $ \preceq_{I} $ the induced relation on $ A $. Then the relational system $ \langle A, \cdot , \rightarrow , 1, \preceq_{I} \rangle $ is a pre-ordered residuated system. 
\proof
Suppose that $ \mathbf{A} $ is a residuated quasi-directoid. 
We firstly prove that $ \preceq_{I} $ is a pre-order on $ A $. Since $ x\sqcup x = x $, then $ x\preceq_{I} x $ for each $ x\in A $, i.e. $ \preceq $ is reflexive.  
For transitivity, suppose that $ a\preceq_{I} b\preceq_{I} c $, we have $ a\sqcup b= b $ and $ b\sqcup c= c $. Therefore:
\begin{align*}
&  a\sqcup c= a\sqcup (b\sqcup c) &&
\\ & = a\sqcup ((a\sqcup b)\sqcup c) &&
\\ & = (a\sqcup b)\sqcup c &&
\\ & =b\sqcup c = c, &&
\end{align*}
hence $ a\preceq_{I} c $.
We still need to check that $ \langle A, \cdot , \rightarrow , 1, \preceq_{I} \rangle $ satisfies conditions 1), 2), 3) of Definition \ref{def: residuated relational systems}. \\
Condition 1) is trivially satisfied. 
Conditions 2) and 3) are direct consequences of axiom f) and g), respectively.  
\endproof
\end{theorem}
We are now going to show that the multiplication for a supremal element is a supremal element for the set of multiples. 
\begin{proposition}\label{prop: moltiplicazione di sup element}
Let $ \alga $ be a pre-ordered residuated system, $ Z\subseteq A $ and $ a\in A $. If $ k $ is a supremal element for $ Z $ then $ a\cdot k $ is a supremal element for the set  $ aZ=\{ a\cdot z: z\in Z\} $. 
\end{proposition}
\begin{proof}
Let $ k $ be a supremal element for $ Z $, then $ z\preceq k $ for each $ z\in Z $, thus by Proposition \ref{pro: aritmetica quasi-order}, $ a\cdot z\preceq a\cdot k $. Assume now that $ a\cdot z\preceq t $, for each $ z\in Z $. Then $ z\preceq a\rightarrow t $ and, since $ k $ is a supremal element for $ Z $, $ k\preceq a\rightarrow t $, whence $ a\cdot k\preceq t $, i.e. the element $ a\cdot k $ is a supremal element for the set $ aZ $.
\end{proof}
\noindent
We recall that any pre-order relation on a set $ A $ generates an equivalence relation as follows. 
\begin{equation}\label{def: Clifford-McLean relation}
 (x,y)\in \theta\;\;  \emph{if and only of}\;\;  x\preceq y \;\; \emph{and}\;\;  y\preceq x .
\end{equation}
\noindent
The equivalence relation above turns out to be very useful to get a poset out of a pre-ordered residuated system. Moreover, notice that relation $ \theta $ can be defined on a  residuated quasi-directoid using equalities, indeed: 
\begin{equation}\label{eq: Clifford-McLean relation by equations}
 (x,y)\in \theta \;\;\; \text{if and only if} \;\;\;  x\sqcup y = y \;\; \text{and} \;\;  y\sqcup x = x.
\end{equation}

\begin{proposition}\label{prop: Clifford-McLean è una congruenza}
Let $ \mathbf{A}=\langle A, \cdot, \to, \sqcup, 1\rangle $ be a residuated quasi-directoid and $ \preceq $ the induced pre-order. Let $\theta$ be the equivalence relation defined in \eqref{eq: Clifford-McLean relation by equations}. If $\theta$ is a congruence on the reduct $\langle A,\sqcup \rangle$, then $ \theta $ is a congruence on $ \mathbf{A} $. 
\proof 
We only need to prove that $\theta$ preserves multiplication and its residuum. Suppose $ (x,y)\in \theta $. It holds $ (x\cdot z, y\cdot z)\in \theta $, as, by Proposition \ref{pro: aritmetica quasi-order}, multiplication preserves the pre-order. As regards the residual, suppose $ (x,y)\in\theta $, then, applying Proposition \ref{pro: aritmetica quasi-order} (b), one gets  $ (x\rightarrow z, y\rightarrow z)\in\theta $ and $ (z\rightarrow x, z\rightarrow y)\in \theta$. 
\endproof
\end{proposition}
\noindent
The importance of relation $ \theta $ is justified by the fact that the quotient $ A/\theta $ turns naturally into a poset. It is indeed folklore that if 
 $ \langle A, \preceq \rangle $ is a pre-ordered set and $ \theta $ the equivalance relation introduced above then the binary relation $ \leq $ defined on $ A/\theta $ by: 
$$ [a]_{_\theta}\leq [b]_{_\theta} \text{ if and only if } a\preceq b $$
for any $ a,b\in A $, is a partial ordering on $ A/\theta $, see for example \cite{Rasiowa63book}. 

It follows from Proposition \ref{prop: Clifford-McLean è una congruenza} and the above observation that it is possible to get a pocrim (see Example 1) as a quotient of a residuated quasi-directoid.
\begin{corollary}\label{cor: un Pocrim da un residuated relational quasi-directoid}
Let $ \mathbf{A}$ be a residuated quasi-directoid and $ \theta $ the equivalence relation defined in \eqref{eq: Clifford-McLean relation by equations}. If $\theta$ is a congruence on the reduct $\langle A, \sqcup\rangle$, then $ \mathbf{A}/\theta $ is a pocrim.  
\end{corollary}
\noindent
We now claim that residuation, in the class of residuated quasi-directoids, can be expressed in terms of identities. 
The candidates to replace residuation are the following:
\begin{itemize}
\item[(a)] $ (x\rightarrow y)\cdot x\preceq y $;
\item[(b)] $ (x\cdot y)\rightarrow z \preceq x\rightarrow (y\rightarrow z) $;
\item[(c)] $ x \rightarrow ( y\rightarrow z)\preceq (x\cdot y)\rightarrow z $;
\item[(d)] $ x\rightarrow (x\sqcup y) \preceq  1 $;
\item[(e)] $ 1 \preceq x\rightarrow (x\sqcup y) $;
\item[(f)] $x\cdot z \preceq (x\sqcup y)\cdot z$.
\end{itemize}

\noindent
It is not difficult to notice that all the above conditions can be expressed by identities, by simply observing that $ x\preceq y $ is equivalent to $ x\sqcup y = y $, for each $ x,y\in A $. 
We can now show that the residuation condition for residuated quasi-directoids can be expressed using identities only. 
\begin{proposition}\label{prop: la residuazione si esprime con equazioni}
Let $ \mathbf{A}=\langle A, \cdot, \rightarrow, \sqcup, 1\rangle $ be an algebra of type $ \langle 2,2,2,0\rangle $ satisfying all the axioms in \emph{Definitions \ref{def: quasi-directoid}} and \emph{\ref{def:pre-ordered residuated quasi-lattices}} with the exception of condition \emph{g)}. Then $ \mathbf{A} $ satisfies axiom \emph{g)} if and only if it satisfies the identities \emph{(a), (b), (c), (d), (e), (f)} above.
\end{proposition}
\proof
For the left to right direction, we just need to show that (a), (b), (c), (d), (e) and (f) hold in any residuated quasi-directoid. In order to get this we simply rely on the fact that (a), (b), (c) hold in any pre-ordered residuated system, by Propositions \ref{prop:aritmetica2} and \ref{pro: aritmetica quasi-order}. Furthermore, (d) is an instance of axiom f) in Definition \ref{def:pre-ordered residuated quasi-lattices}. As regards (e), $ 1\cdot x =x\preceq x\sqcup y $ by Definition \ref{def: quasi-directoid}, hence by residuation $ 1\preceq x\rightarrow (x\sqcup y) $. Finally, (f) follows from the monotonicity of multiplication and the fact that $x\preceq x\sqcup y$.\\
For the converse, we have to derive the residuation condition g) using equations (a), (b), (c), (d), (e) and (f). At first, we observe that (f) implies that multiplication preserves the induced pre-order. Suppose $ a\cdot b\preceq c $, then $ (a\cdot b)\sqcup c = c $. 
By (e), $ 1\preceq a\cdot b\rightarrow ((a\cdot b)\sqcup c) = a\cdot b\rightarrow c\preceq a\rightarrow (b\rightarrow c) $, by (b). Thus $ a=1\cdot a\preceq (a\rightarrow (b\rightarrow c))\cdot a\preceq b\rightarrow c $, by (a), hence $ a\preceq b\rightarrow c $ (in the first inequality we have used that $1\preceq a\to (b\rightarrow c$). \\
Suppose now that $ a\preceq b\rightarrow c $, i.e. $ a\sqcup (b\rightarrow c) = b\rightarrow c $. By (e) $ 1\preceq a\rightarrow (a\sqcup (b\rightarrow c))= a\rightarrow (b\rightarrow c)\preceq a\cdot b\rightarrow c $ by equation (c). Hence $ a\cdot b= 1\cdot (a\cdot b)\preceq (a\cdot b\rightarrow c)\cdot (a\cdot b)\preceq c $ by equation (a), thus $ a\cdot b\preceq c $ (in the first inequality we have used that $1\preceq a\cdot b\rightarrow c $). 
\endproof       

\begin{corollary}\label{cor: varietà}
The class of residuated quasi-directoids forms a variety. 
\end{corollary}

\section{Residuated systems with negation}\label{sec:3.3}

In what follows we expand the language of residuated relational systems, adding a new constant $ 0 $. 
\begin{definition}\label{def: sistema con 0}
A \emph{residuated relational system with 0} is a structure $ \alga = \langle A, \cdot, \rightarrow, 0,1, R \rangle $ such that $ \alga = \langle A, \cdot, \rightarrow,1, R \rangle $ is a residuated relational system and $0\in A$ is a constant such that $ 0\cdot x = x\cdot 0 = 0 $ for each $x\in A $.   
\end{definition}
In a residuated relational system with 0 it makes sense to define a new operation as $ x':= x\rightarrow 0 $. Such operation will be referred to as \emph{negation}. For sake of simplicity we will write $ x'' $ as an abbreviation for $ (x')' $. 
\begin{lemma}\label{lem: 0Rx}
Let $ \alga $ be a residuated relational system with 0, whose relation $R$ is reflexive. Then $ (0,y)\in R $ for each $ y\in A $.
\end{lemma}
\begin{proof}
From Proposition \ref{prop:aritmetica2} we have $ ((x\rightarrow y)\cdot x, y)\in R $. Setting $ x = 0 $ we get $ (0,y) = ((0\rightarrow y)\cdot 0, y)\in R $.
\end{proof}
Here are some basic facts concerning negation in residuated relational systems with 0. 
\begin{proposition}\label{prop: negazione}
Let $ \alga $ be a residuated relational system with 0 and $ R $ a reflexive relation. Then
\begin{itemize}
\item[(a)] $(0,1')\in R $ and $ (1',0)\in R $
\item[(b)] $ (1,0')\in R $ and $ (0',1)\in R $
\item[(c)] $ (x,x'')\in R $
\item[(d)] $ (x\cdot x',0)\in R $,
\end{itemize}
for all $x\in A$.
\end{proposition}
\begin{proof}
(a) follows from Proposition \ref{prop:aritmetica2} (d) setting $ x = 0 $. \\
(b) By Proposition \ref{prop:aritmetica2} (a) (with $ x=0 $) we obtain $ (1,0') = (1, 0\rightarrow 0)\in R $. Clearly $ (0',1)\in R $ holds by the definition of a residuated relational system. \\
(c) a is direct consequence of Proposition \ref{prop:aritmetica2} (e) by setting $ y = 0 $. \\
(d) follows from Proposition \ref{prop:aritmetica2} (b), setting $y=0 $. 
\end{proof}
\begin{proposition}\label{prop: R antisimmetrica con 0}
If $ \alga $ is a residuated relational system with 0 and $ R $ reflexive and antisymmetric, then $ 0\rightarrow y = 1 $  for each $ y\in A $.
\end{proposition}
\begin{proof}
By Lemma \ref{lem: 0Rx}, $ (0,y)\in R $ and since $ R $ is antisymmetric then, by Proposition \ref{prop: definitione equazioanle di R}, $ 0\rightarrow y = 1 $. 
\end{proof}
The above facts lead to the following
\begin{corollary}\label{cor: 0=1'}
Let $ \alga=\langle A, \cdot, \rightarrow, 0,1, R \rangle $ be a residuated relational system with $0$ and $ R $ a reflexive and antisymmetric relation. Then $ 0'=1 $, $1'=0$ and $ x\cdot x'=0 $. 
\end{corollary}
\begin{proof}
The first two assertions follow from Proposition \ref{prop: negazione}, while the third follows from Lemma \ref{lem: 0Rx} and Proposition \ref{prop: negazione}.
\end{proof}
The following proposition states the properties of negation in pre-ordered residuated systems with $0$. These are residuated relational systems with 0 which are pre-ordered residuated systems, too.
\begin{proposition}\label{prop: negazione nei pre-ordinati}
Let $ \alga $ be a pre-ordered residuated system with 0. Then
\begin{itemize}
\item[(a)] if $ x\preceq y$ then $ y'\preceq x' $
\item[(b)] $ x\preceq x'' $, $x'\preceq x'''$, $ x'''\preceq x' $
\item[(c)] $ (x\cdot y)'\preceq x\rightarrow y' $ and $x\rightarrow y' \preceq (x\cdot y)' $
\item[(d)] $ x\rightarrow y'\preceq y\rightarrow x' $ and $ y\rightarrow x'\preceq x\rightarrow y' $
\item[(e)] $ (x\rightarrow y)\cdot y'\preceq x' $
\item[(f)] $ x\rightarrow y\preceq y'\rightarrow x' $
\end{itemize}
\end{proposition}
\begin{proof}
(a) follows from Proposition \ref{pro: aritmetica quasi-order} (b), upon setting $ z = 0 $. \\
(b) $ x\preceq x'' $ is a consequence of Proposition \ref{prop:aritmetica2} (e) setting $ y = 0 $. $x'\preceq x'''$ is obtained from $x\preceq x''$ by substituting $ x $ by $ x' $, while $x'''\preceq x'$ follows from $x\preceq x''$ by (a).  \\
(c) follows from Proposition \ref{pro: aritmetica quasi-order} (d) and (e) with $ z = 0 $. \\
(d) follows from Proposition \ref{pro: aritmetica quasi-order} (f) by setting $ z = 0 $. \\
(e) follows from Proposition \ref{pro: aritmetica quasi-order} (g) by setting $ z = 0 $. \\
(f) follows from Proposition \ref{pro: aritmetica quasi-order} (i) by setting $ z = 0 $.
\end{proof}
\noindent
One can observe that the properties of negation listed in Proposition \ref{prop: negazione nei pre-ordinati} correspond to those of negation in intuitionistic logic.
\begin{definition}\label{def: doppia negazione}
Let $ \alga $ be a pre-ordered residuated system with 0. We say that $ \alga $ satisfies the \emph{law of double negation} if $ x=x'' $, for each $ x\in A $, i.e. if $ x= (x\rightarrow 0)\rightarrow 0 $. 
\end{definition}

\begin{proposition}\label{prop: con la doppia negazione}
Let $ \alga $ be a pre-ordered residuated system with $0$ satisfying the law of double negation. Then
\begin{itemize}
\item[(i)] $ x\rightarrow y\preceq (x\cdot y')' $ and $(x\cdot y')'\preceq x\rightarrow y $.
\item[(ii)] $ y'\rightarrow x'\preceq x\rightarrow y $.
\end{itemize}
\end{proposition}
\begin{proof}
(i) Using the law of double negation and condition (e) in Proposition \ref{pro: aritmetica quasi-order}, we have $ x\rightarrow y= x\rightarrow ((y\rightarrow 0)\rightarrow 0)\preceq x\cdot (y\rightarrow 0)\rightarrow 0 = (x\cdot y')' $. Similarly for the other claim we use the law of double negation and condition (d) in Proposition \ref{pro: aritmetica quasi-order}, obtaining $(x\cdot y')'= x\cdot (y\rightarrow 0)\rightarrow 0 \preceq x\rightarrow ((y\rightarrow 0)\rightarrow 0) = x\rightarrow y $. \\
(ii) The result follows from Proposition \ref{prop: negazione nei pre-ordinati} (f), where, using the law of double negation, we get: $ y'\rightarrow x'\preceq x''\rightarrow y'' = x\rightarrow y $. 

\end{proof}

\begin{center}
\textbf{Acknowledgments}
\end{center}

The work of the first author is supported by the Italian Ministry of Scientific Research (MIUR) for the support within the PRIN project `Theory of Rationality: logical, epistemological and computational aspects'. The research of the second author is supported by the project IGA PrF 2014016 Palacky University Olomouc and by the Austrian Science Fund (FWF), project I 1923-N25, and the Czech Science Foundation (GA\v{C}R): project 15-34697L. Finally, we also thank Francesco Paoli for his suggestion to work on the topic and an anonymous referee for his/her valuable comments on a previous draft.  


\end{document}